\documentclass[12pt]{article}

\usepackage{amsmath,amsfonts,amssymb,graphics,amsthm}
\usepackage{subfigure}
\usepackage{geometry} 
\usepackage{nicefrac}
\usepackage[pdfborder={0 0 0},colorlinks=true,linkcolor=black,citecolor=black,urlcolor=blue]{hyperref}
\usepackage{youngtab}
\usepackage{caption}
\allowdisplaybreaks

\newtheorem{thm}{Theorem} 
\newtheorem{lemma}{Lemma}

\def\P{\mathbb{P}}
\def\E{\mathbb{E}}
\def\C{\mathbb{C}}
\def\R{\mathbb{R}}

\def\eps{\varepsilon}

\DeclareMathOperator{\Tr}{Tr}

\newcommand{\indic}[1]{\mathbf{1}_{\{#1\}}}

\begin{document}

\title{Cycle structure of the interchange process and representation theory}

\author{Nathana\"el Berestycki\thanks{DPMMS, Cambridge
    University. Supported in part by EPSRC grants EP/GO55068/1 and
    EP/I03372X/1.} \and Gady Kozma\thanks{Weizmann
    Institute. Supported in part by the Israel Science Foundation}}

\date{}

\maketitle

\begin{abstract}
  Consider the process of random transpositions on the complete graph
  $K_n$. We use representation theory to give an exact, simple
  formula for the expected number of cycles of size $k$ at time $t$,
  in terms of an incomplete Beta function. Using this we show that the
  expected number of cycles of size $k$ jumps from 0 to its
  equilibrium value, $1/k$, at the time where the giant component of
  the associated random graph first exceeds $k$. Consequently we
  deduce a new and simple proof of Schramm's theorem on random
  transpositions, that giant cycles emerge at the same time as the
  giant component in the random graph. We also calculate the
  ``window'' for this transition and find that it is quite thin. Finally, we give a new proof of a result by the first author and Durrett that the random transposition process exhibits a certain slowdown transition.
  The
  proof makes use of a recent formula for the character decomposition of the number of cycles of a given size in a permutation, and the Frobenius formula for the character ratios.
\end{abstract}


\section{Introduction and main results}

Consider the complete graph $K_n$ on $n$ vertices, and let $\sigma_t$
be the random walk on $S_n$ that results when considering the
interchange process on $K_n$. That is, $\sigma_t$ is the usual random
transposition process (see e.g.\ \cite{Schramm}) sped up by a factor
$\binom n2$.

Our first result gives an exact and surprisingly simple formula for the expected number of cycles of size $k$ at time $t$.

\begin{thm} \label{T:exact}
  Fix any $1\le k \le n$ and let $s_k(t)$ be the number of cycles of
  size $k$ at time $t$ in $\sigma_t$. Then
  \[
    \E(s_k(t)) =
      \binom nk \left[ \frac1k x\phi(x) + \int_x^1 \phi(y) dy
        \right]
,
  \]
  where $\phi(y) = y^{n-k}(1-y)^{k-1} $ and $x= e^{-tk}$.
  \end{thm}

This integral is known as the incomplete beta function (the
integral from $0$ to $1$ is the regular beta function).
The proofs of Theorem \ref{T:exact} is based on representation
theory. The key argument is a formula of Gil Alon and one of us \cite{AlonKozma} for the character decomposition of the number of cycles of a permutation, as well as Frobenius' formula for the values of the character ratios.

Our second result uses the above formula to show that in the limit as
$n \to \infty$, the quantity $\E(s_k(t))$ exhibits a sharp transition
from the value $0$ to $1/k$ at a time $t_{n,k}$ which is essentially
$(-1/k) \log(1-k/n)$. We note that this time is an order magnitude
smaller than the mixing time for $\sigma_t$ which (with this
parametrization) is $(\log n)/n$ (see \cite{DiaconisSh} or
\cite{BSZ}). The width of this transition is shown to be order
$1/n^{3/2}$ when $k$ is of order $n$ (which corresponds to a width of
order $\sqrt{n}$ in the traditional scaling of random
transpositions). This is reminiscent of the cutoff phenomenon for the
mixing time, see \cite{DiaconisSh}, \cite{LS} or \cite[\S 18]{LPW09} for
the cutoff phenomenon.

\begin{thm}
  \label{T:trans}
  Let $1\le k < n$ and let $t_{n,k}$ be the unique $t$ such that $e^{-kt} = (n-k)/(n-1).$ 
  Then
\[
  \bigg|\E(s_k(t)) - \frac 1k\indic{t>t_{n,k}} \bigg| \le Cq \exp\left\{ - c(n-k)\min\{|t-t_{n,k}|^2k^2,1\} \right\}
  \]
  where $q=q_n(k)$ is a polynomial factor, $q=n^{3/2}k^{-3/2}(n-k)^{-1/2}$.
\end{thm}
\noindent(Here and below $c$ and $C$ pertain to
  unspecified positive universal constants, possibly different from
  one place to another).

\begin{figure}[t]
\begin{center}
  \includegraphics{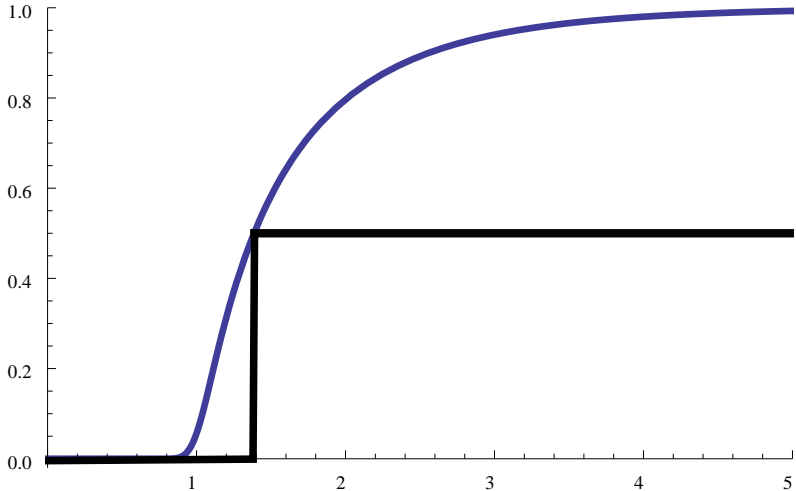}
\end{center}
  \caption{Approximate plots of $ (n^2/k) \E(s_k(t/n))$,  
     together with the relative size of the giant component, for $n = 200$ and $k=100$. The scaling factor in front is chosen so that if $k /n \to \alpha$ then the limiting step function takes the values 0 and $\alpha$.}
  \label{fig}
\end{figure}

At first sight it might seem as if Theorem \ref{T:trans} cannot
possibly hold. After all, a large cycle of $\sigma_t$ (say of size $\frac 13 n$)
is necessarily contained in the giant component of a corresponding
Erd\H{o}s-R\'enyi graph $G(n,t)$ with edge density $t$ (see
e.g.\ \cite{Schramm}). It is easy to check that $t_{n,k}$ is the first time that the giant component has a relative size which exceeds $1/3$, so it is clear that $\E(s_k(t))$ must be close to 0 before $t_{n,k}$. What Theorem \ref{T:trans} says is that $\E(s_k(t))$ suddenly reaches its equilibrium value precisely at that time, and does not change afterward. Note however that as $t$ increases above $t_{n,k}$, the cluster continues to
grow. (See Figure \ref{fig}). So how come the probability for a cycle of size exactly $\frac
13 n$ does not grow with it? The answer lies in examining a completely
random permutation. The expected number of cycles of length $k$ in a
random permutation of $n$ elements is exactly $\nicefrac 1k$ and hence
{\em does not depend on $n$}, of course, as long as $n\ge k$. This
explains why the size of the giant component does not affect
$\E(s_k)$, as long as it is bigger than $k$.

The proof of Theorem \ref{T:trans} follows by estimating the terms in
Theorem \ref{T:exact} carefully. Let us just remark how the phase
transition comes about. The integrand in Theorem \ref{T:exact}
(i.e.\ $\phi$) has a peak at exactly the critical value $t_{n,k}$. So
in the subcritical case (i.e.\ $t$ small so $x$ large), $\int_x^1\phi$
can be estimated simply by $\phi(x)$, which is small. In the supercritical case (i.e.\ $x$ small), the integral
is estimated by writing it as $\int_0^1-\int_0^x$. The
integral from $0$ to $1$ is easy to calculate, and gives the
$\nicefrac 1k$ term, and the integral from 0 to $x$ is the error term.

From this we can derive a new proof of the result, first proved by Schramm \cite{Schramm}, that giant cycles (of macroscopic size) emerge at time $1/n$ with high probability asymptotically as $n \to \infty$.

\begin{thm}
  \label{T:gc}
  Let $c>1$ and let $t=c/n$. Let $C(t)$ denote the length of the largest cycle of $\sigma_t$. Then as $\eps \to 0$,
  $$
  \lim_{\eps \to 0} \liminf_{n \to \infty }\P( C(t) > n \eps ) =1.
  $$
\end{thm}

Schramm's results are stronger and also give information about the
{\em joint} distribution of cycles, e.g.\ what is the probability that
the largest cycle is in $[0.6n,0.7n]$ and the second largest is in
$[0.1n,0.2n]$. We will not discuss the joint distribution in this paper.

Theorem \ref{T:gc} is the mean-field case of a well-known conjecture
of B\'alint T\'oth \cite{toth} that the cycle structure of the
interchange process on the graph $\mathbb{Z}^d, d \ge 3$, exhibits a
phase-transition with infinite cycles appearing at some finite time
(see e.g. \cite{toth} for definitions). This paper was initially
motivated by the desire to see if it would be possible to use a
representation-theoretic approach to this conjecture. Our proof of
Theorem \ref{T:gc} can be seen as the mean-field case of this
programme (see also Theorem 1 of \cite{NB} for a possible alternative
route, which yields a slightly weaker conclusion).  Let us remark that
the recent resolution of Aldous' spectral gap conjecture
\cite{CLR} could aid one in an algebraic attack on the non-mean-field
case.

Interestingly, the asymptotics leading to the transition in Theorems \ref{T:trans} and \ref{T:gc} do \emph{not} come from the contribution of a single representation of $S_n$ with all other contributions being negligible. Rather, it is the result of some very remarkable cancelations when summing up over all representations of $S_n$. These cancelations lead to the simple formula of Theorem \ref{T:exact}. Here it is essential to have an exact computation that keeps track of \emph{all} representation in order to establish the emergence of giant cycles of Theorem \ref{T:gc}.  Nevertheless, if one wants only weaker results, e.g.\ that whenever
$t>100 t_{n,k}$ we already have $|\E(s_k(t))-\nicefrac 1k|$
exponentially small, then in this case a single representation (the
trivial one, $[n]$) gives the main contribution, with the contribution
of all the others combined negligible. In other words, ignoring the
cancelation between representations will cause the information about
the sharpness of the phase transition to be lost, but one may still
show the existence of an ``ordered phase''.

Finally, we show how our formula allows us to recover a result of Berestycki and Durrett \cite{BD} concerning a slowdown transition for the interchange process. Let $d(t)$ denote the distance between $\sigma_t$ and $\sigma_0$, i.e., the minimal number of transpositions needed to write $\sigma_t$ as a product of transpositions. (If one views $\sigma_t$ as a random walk on the Cayley graph of $S_n$ generated by the set of transpositions, then $d(t)$ is the graph distance between $\sigma_t$ and $\sigma_0$).

\begin{thm}
  \label{T:slowdown}
  Let $c>0$ and let $t = c/n$. Then as $n \to \infty$,
  \begin{equation}\label{E:slowdown}
  \frac1n\E(d(t)) \to u(c): = 1- \sum_{k=1}^\infty \frac{k^{k-2}}{k!} \frac1c (ce^{-c})^k.
\end{equation}
\end{thm}

The asymptotic behaviour \eqref{E:slowdown} was first obtained in Theorem 3 of \cite{BD}. It is known that for $c<1$, the function $u(c)$ simplifies to $u(c) = c/2$ while for $c>1$, $u(c) < c/2$. The function $u$ is $C^\infty$ everywhere except at $c=1$, for which $u''(1) = - \infty$. This result is indicative of a slowdown transition, in the following sense: so long as $c<1$, the random walk escapes from the origin at maximum speed (i.e., all but a negligible fraction of moves take it away from the origin), while for $c>1$ the random walk starts decelerating brutally --- indeed, the acceleration is $-\infty$ at $c=1$. The methods of \cite{BD} relied heavily on the use of an associated Erd\H{o}s-Renyi random graph process, and more precise results were obtained (e.g., a central limit theorem for $d(t)$). However we will see that the formula \eqref{E:slowdown} is obtained with essentially no effort from our Theorem \ref{T:exact}.

It is interesting to note that Theorem \ref{T:slowdown} is derived by
applying Theorem \ref{T:exact} for constant $k$, unlike Theorems
\ref{T:trans} and \ref{T:gc} for which the relevant regime is
$k\asymp n$. It would be interesting to find applications of our
formula in other regimes of $k$, for example $k=n-o(n)$.

\section{Proofs}

\subsection{Preliminaries}

For a permutation $\sigma$, let $\alpha_k(\sigma)$ denote the number of cycles of size $k$ of $\sigma$, and $s_k(t) = \alpha_k(\sigma_t)$. For a representation $\rho:S_n \to GL(\mathbb{C}^{\dim \rho})$ of $S_n$, let $\chi_\rho$ be its character, i.e., $\chi_\rho(\sigma) = \Tr( \rho(\sigma)).$
Note that $\alpha_k( \sigma)$ is a class function. Hence, since characters form an orthonormal basis of class functions (see, e.g., Proposition 1.10.2 in \cite{sagan}) there exists $a_\rho \in \C$ such that
\begin{equation}\label{D:arho}
\alpha_k(\sigma) = \sum_\rho a_\rho \chi_\rho(\sigma)
\end{equation}
where the sum is over irreducible representations $\rho$.
In \cite{AlonKozma} the value of $a_\rho$ for all representations
$\rho$ was worked out. To describe this result we need to recall some
standard background in the representation theory of the symmetric group. The irreducible
representations of $S_n$ are parametrized by Young diagrams of size
$n$, that is, sequences $\lambda = [\lambda_1, \ldots, \lambda_j]$
with $\lambda_1 \ge \cdots \ge \lambda_j \ge 1$ and $\sum_i \lambda_i
= n$, which may be thought of as a collection of boxes sitting on top
of one another, with $\lambda_i$ boxes in row $i$. See \cite{sagan}
for the details of this parameterization (we will not use it directly in this paper).

As it turns out, there are very few Young diagrams for which the
corresponding coefficient $a_\rho$ defined by \eqref{D:arho} is
nonzero. To state things as simply as possible, we restrict ourselves
to the case where $k \le n/2 $ (the other case is similar and we come back to this later).
Consider the representation $\rho_i = [n-k, k-i, 1^i]$ where $1^i$ indicates that there are $i$ rows with exactly one box. Then Theorem 3 of \cite{AlonKozma} states that $a_\rho$ is zero unless $\rho=\rho_i$ for some $0\le i \le k-1$, in which case
$a_\rho$ is simply equal to $(-1)^i/k$, or $ \rho = [n]$ the trivial representation, in which case $a_\rho = 1/k$.

\begin{lemma}
  \label{L:repres}
For a representation $\rho$, let $r(\rho) = \chi_\rho(\tau)/d_\rho$ be
the character ratio of $\rho$, i.e., when $\tau$ is any transposition,
$\chi_\rho(\tau) = \Tr(\rho(\tau))$ is the character of $\rho$ at $\tau$, and $d_\rho$ is the dimension of $\rho$. Then
\begin{equation}\label{E(sk)}
  \E(s_k(t)) = \sum_\rho a_\rho d_\rho \exp\bigg\{ {n \choose 2} (r(\rho)-1) t\bigg\}.
\end{equation}
\end{lemma}
\begin{proof}
To compute $s_k(t)$, let $Q$ be the infinitesimal generator of the random walk. That is, viewing $Q$ as an element of the group algebra $\mathbb{C}[S_n]$, $Q = \sum_{i<j} [(i,j) - \text{id}]$ (here, $(i,j)$ denotes the transposition of $i$ and $j$ and $\text{id}$ is the identity permutation). Thus, $\P(\sigma_t = s) = e^{tQ}(o, s) = f(s),$ say. Then
\begin{align}
  \E(s_k(t)) & = \sum_{s\in S_n} \alpha_k(s) \P(\sigma_t = s) \nonumber \\
  & = \sum_\rho \sum_{s \in S_n} a_\rho \chi_\rho(s) f(s) \nonumber \\
  & = \sum_\rho a_\rho \Tr\left(\sum_{s \in S_n} \rho(s) f(s)\right) \nonumber \\
  & = \sum_\rho a_\rho \Tr \widehat f(\rho), \label{skfourier}
\end{align}
where $\widehat f(\rho)$ denotes the Fourier transform of $f$ at $\rho$. Note that $f(s) = \sum_{k=0}^\infty \frac{1}{k!}t^k Q^{k}(s) $.
Since convolutions become (matrix) products after taking Fourier transform, we get
\[
\widehat f(\rho) = \sum_{k=0}^\infty \frac{t^k}{k!} \widehat Q(\rho)^k
 = \exp( t \widehat Q(\rho))
\]
where the second $\exp$ is a matrix exponential. Now note that
$Q(s)$ is constant on conjugacy classes, i.e., is a class function (here is where we use that we do the interchange process on the complete graph, and not on any other graph). It easily follows from Schur's lemma that $\widehat Q(\rho) = \lambda_\rho I$ for some $\lambda_\rho \in \R$, and thus we conclude that 
\begin{equation}\label{Fourier_f}
\widehat f(\rho) = e^{t\lambda_\rho}I .
\end{equation}
To obtain the constant $\lambda_\rho$ we proceed as follows: on the one hand we have $\Tr \widehat Q(\rho)=d_\rho\lambda_\rho$, so $\lambda_\rho = (1/d_\rho) \Tr \widehat Q(\rho)$. On the other hand, $Q=\sum_{i<j}[(ij)-\text{id}]$ so by linearity of the trace,
\[
\Tr \widehat Q(\rho)=\sum_{i<j}\chi_\rho((ij))-\chi_\rho(\text{id})={n \choose 2}(\chi_\rho(\tau)-d_\rho),
\]
where $\chi_\rho(\tau)$ denotes the character of $\tau$ evaluated at an arbitrary transposition (again, it is not important which one, since characters are class functions). Hence, dividing by $d_\rho$, we get 
\begin{equation}\label{lambda_rho}
\lambda_\rho = { n \choose 2} (r(\rho)-1)
\end{equation}
where $r(\rho) = \chi_\rho(\tau)/ d_\rho$ is the character ratio at a transposition.
Combining \eqref{skfourier}, \eqref{Fourier_f} and \eqref{lambda_rho}, we have arrived at
$$
  \E(s_k(t)) = \sum_\rho a_\rho d_\rho \exp\bigg\{ {n \choose 2} (r(\rho)-1) t\bigg\}
$$
as needed.
\end{proof}

\subsection{Proof of Theorem \ref{T:exact}, case \texorpdfstring{$k
    \le n /2 $}{k less than n/2}.}

We now explain how to compute the various terms in the sum involved in \eqref{E(sk)} in order to get to the exact formula of Theorem \ref{T:exact}. We assume for now $k \le n/2$.

\begin{proof}
We start by recalling Frobenius's formula, which gives the values of the character ratios for any representation. Suppose $\rho$ is parametrized by the Young tableau
$(\rho_1, \ldots, \rho_j)$, then (see \cite[Equation $(D-2)$, p.\ 40]{Diaconis} or \cite[Lemma 7]{DiaconisSh}),
\begin{equation}
\label{Frobenius}
r(\rho) = \frac1{n(n-1)} \sum_{i=1}^j \rho_i^2 - (2i-1) \rho_i.
\end{equation}
Thus if $\rho= \rho_i$ is the representation given by $\rho_i = [n-k, k-i, 1^i]$, $0\le i \le k-1$, then
\begin{align}
  n(n-1) r(\rho_i) &= (n-k)^2 - (n-k) + (k-i)^2 - 3(k-i) \nonumber \\
  &  \ \ \ +1 - 5 + \ldots + 1 - (2(i+2) -1)\nonumber \\
  & =  (n-k)^2 - (n-k) + k^2 - 2i k + i^2  - 3k + 3i  - (i^2 + 3i)\nonumber \\
  & = (n-k)^2 - n - 2k + k^2  - 2ik.\label{ratio_value}
\end{align}
Notice the cancelation of the terms in $i^2$, which is crucial for the following calculations.

\begin{figure}
\input{hook.pstex_t}
\caption{Sizes of the hooks for the diagram $[n-k,k-i,1^i]$.}\label{fig:hook}
\end{figure}
By the hook-length formula (see e.g. (4.12) in \cite{FultonHarris}),
it is also easy to compute $d_{\rho_i}$, the dimension of the
representation $\rho_i$ (see Figure \ref{fig:hook}):
\begin{equation}
  \text{dim}(\rho_i) = \frac{n! (n-2k+i+1)}{i!  k (n-k)! (k-i-1)!(n-k+i+1)}.\label{dimrho_i}
\end{equation}

We now plug these various expressions into \eqref{E(sk)}. Recall that
$a_\rho = (-1)^i/k$ when $\rho = \rho_i$, and $a_\rho = 1/k$ when $
\rho$ is the trivial representation $[n]$. We obtain the following
expression, after rearranging the terms (mostly to isolate the
$i$-independent terms and take them out of the sum). Let $m = k-1$,
and $x = \exp( - kt)$. Then
\begin{align}
  \E(s_k(t)) & = \frac1k + C_{n,k} \sum_{i=0}^m \frac{m! }{i! (m-i)!} \frac{n-2k+i+1}{n-k+i+1} (-1)^i  x^i  \label{E(sk)2}
\end{align}
where
\begin{align}
C_{n,k} &= \frac1k \frac{n!}{k! (n-k)!} \exp\left\{ t {n \choose 2}\left( \frac{ (n-k)^2 - n - 2k + k^2}{n(n-1)}-1\right)\right\} \nonumber\\
&= \frac1k {n \choose k} \exp\left\{ \frac{t}2 \left( n^2 - 2nk + k^2 - n - 2k + k^2\right) - t\frac{n^2 - n}2 \right\} \nonumber \\
& = \frac1k {n \choose k} \exp\left\{ - t k (n+1 - k)\right\}
=\frac 1k \binom{n}{k}x^{n-m}.\label{Cnk}
\end{align}
To compute the right-hand side of \eqref{E(sk)2}, note that $\sum_{i=0}^m \frac{m! }{i! (m-i)!} (-1)^i  x^i = (1- x)^m$, hence
\begin{align*}
\E(s_k(t)) & = \frac1k + C_{n,k} \sum_{i=0}^m \frac{m! }{i! (m-i)!} \frac{n-2k+i+1}{n-k+i+1} (-1)^i  x^i  \\
& = \frac1k + C_{n,k} \sum_{i=0}^m {m \choose i} (-x)^i \left( 1- \frac{k}{n-k+i+1}\right)\\
& = \frac1k + C_{n,k} (1- x)^m - C_{n,k} k \sum_{i=0}^m { m \choose i}  (-x)^i \frac1{n-m+i}\\
& = \frac1k + A_1 - A_2,
\end{align*}
say.

This can be simplified a little bit. Observe indeed that
\begin{align*}
  A_2 & = kC_{n,k}\sum_{i=0}^m { m \choose i}  (-x)^i \frac1{n-m+i} \\
  & = k C_{n,k} x^{ - (n - m)} \sum_{i=0}^m {m \choose i}(-1)^i \frac{x^{n-m+i}}{n-m+i }\\
  & = {n \choose k} \sum_{i=0}^m {m \choose i}(-1)^i \int_0^x y^{n-m+i-1} dy\\
  & = {n \choose k} \int_0^x y^{n-k} (1-y)^m dy
\end{align*}
On the other hand, $A_1 = \frac1k { n \choose k} (1-x)^m x^{n-m}$, so
we are led to the following exact expression for $\E(s_k(t))$: with $x
= e^{-t k}$,
\[
  \E(s_k(t)) = \frac 1k+\binom nk \left[ \frac1k x^{n-m}(1-x)^{m}-\int_0^xy^{n-k}(1-y)^m\,dy\right].
\]
Recalling the notation $\phi$ from the statement of the theorem, we
can rewrite this shortly as
\begin{equation}\label{E:from0}
\E(s_k(t))=\frac 1k + \binom nk \left[\frac 1k
  x\phi(x)-\int_0^x\phi(y)\,dy\right].
\end{equation}
The theorem is basically finished: we only need the following simple
transformation:
\begin{lemma}\label{lem:0xx1}
\[
\binom nk \int_0^x\phi(y)\,dy=\frac 1k - \binom nk
\int_x^1\phi(y)\,dy
\]
\end{lemma}
\begin{proof}
Denote
\[
I(x)=\int_x^1\phi(y)=\int_0^1\phi(y)\,dy-\int_0^x\phi(y)\,dy.
\]
The integral from 0 to 1 is just the Beta integral, so
\[
\int_0^1 y^{n-k}(1-y)^{k-1} dy = \frac{\Gamma(n-k+1) \Gamma(k)}{\Gamma(n-k+1 + k)}.
\]
For completeness, here is a short
proof: Fubini's theorem shows that for every $x,y >0$,
$\Gamma(x) \Gamma(y) = \int_0^\infty \int_0^\infty e^{-u -v} u^{x-1} v^{y-1} \,du\, dv$, hence after a change of variables $z = u + v$, $t = u/z$, so that $du\, dv = z\, dz\, dt$, we get
\begin{align*}
\Gamma(x) \Gamma(y) &= \int_0^\infty  e^{-z} z^{x+ y - 1} dz \int_0^1 t^{x-1} (1-t)^{y-1} dt .
\end{align*}
Hence $ \int_0^1 t^{x-1} (1-t)^{y-1} dt  = \Gamma(x) \Gamma(y)
/\Gamma(x+y),$ as required.

Multiplying by $\binom nk$ gives the $\frac 1k$, showing the lemma. \end{proof}
Thus the $\frac 1k$ from lemma \ref{lem:0xx1} cancels
with the $\frac1k$ in (\ref{E:from0}) and leads us to
\begin{equation}
  \label{E(sk)simple}
  \E(s_k(t)) = {n \choose k} \left[ \frac1k x^{n-m}(1-x)^{k-1} + I(x)\right]
\end{equation}
(with $I$ as in the proof of lemma \ref{lem:0xx1}). This finishes the proof of Theorem \ref{T:exact} in the case $k \le n/2 $.
\end{proof}


\subsection{Proof of Theorem \ref{T:exact}, case \texorpdfstring{$k >
    n/2$}{k bigger than n/2}.}
The calculation in this case is very similar to the one in the previous case. In fact, we find the similarity eerie and needing of explanation. What we will show is that (\ref{E(sk)2}) holds, which means that all the calculation from (\ref{E(sk)2}) on holds as is, including the conclusion of the theorem. So we only need to show (\ref{E(sk)2}).

Recall from \cite{AlonKozma} that if $k > n/2$, the representations that arise in the character decomposition of $\alpha_k$ are of two types: for $0\le  i \le 2k - n -2$, we have $\rho_i = [k-i-1, n-k+1, 1^i]$, for which $a_i = (-1)^{i+1}/k$, while for $2k - n \le i \le k-1$, we have $\rho_i = [n-k, k-i, 1^i]$ coming with a coefficient $a_i = (-1)^i/k$. Note in particular, that for $2k- n \le i \le k-1$, the representation and the coefficient is formally identical to those used in the case $k \le n/2$, and hence the corresponding dimension and character ratio are still given  (formally) by the same expressions \eqref{ratio_value} and \eqref{dimrho_i}. Note also that there is no representation associated with $i = 2k-n-1$.

For $0 \le i \le 2k - n -2$, consider the representation $\rho_i = [k-i-1, n-k+1, 1^i]$ for which the coefficient is $a_{\rho_i} = (-1)^{i+1}/k$. Then a calculation using Frobenius' character formula shows (still keeping the notation $m=k-1$),
\begin{align*}
  n(n-1) r(\rho_i) &= (m-i)^2 - (m-i) + (n-m)^2 - 3(n-m) \\
  &  \ \ \ +1 - 5 + \ldots + 1 - (2(i+2) -1)\\
  & = (m-i)^2 - (m-i) + (n-m)^2 - 3(n-m) - (i^2 + 3i) \\
  & = -2ik + (n-k)^2 + k^2 -2k - n
\end{align*}
Comparing with \eqref{ratio_value}, observe that this is exactly the same expression.

\begin{figure}
\input{hook2.pstex_t}
\caption{Sizes of the hooks for the diagram $[k-i-1,n-k+1,1^i]$}\label{fig:hook2}
\end{figure}
It remains to compute $d(\rho_i)$. In order to do so we use again the
hook length formula, which gives us the following (see Figure \ref{fig:hook2}):
$$
\dim(\rho_i) = \frac{n! (2k - n - i -1)}{i! k (n-k)! (k-i -1)! (n-k+i+1)}.
$$
Formally, this is once again exactly the same result as in case $k \le n/2$ but with a minus sign. But since in this case $a_i = (-1)^{i+1}/k$ instead of $(-1)^i/k$, the two minus signs cancel out and the formula which computes $\E(s_k(t))$ is \emph{exactly the same}, except for the fact that there is no term corresponding to $i = 2k - n-1$. However, observe that the term in (\ref{E(sk)2}) corresponding to $i=2k-n-1$ is zero, so it does nothing to the sum when we keep it.

This shows that (\ref{E(sk)2}) holds also when $k>n/2$ and we conclude that, even when $k > n/2$,
\begin{equation*}
  \label{E(sk)simple2}
  \E(s_k(t)) = {n \choose k} \left[ \frac1k x^{n-m}(1-x)^{k-1} + I(x)\right].
\end{equation*}
This finishes the proof of Theorem \ref{T:exact}.\qed


\subsection{Proof of Theorem \ref{T:trans}}

We need to estimate the terms appearing in Theorem \ref{T:exact}. For
this we denote $B_1 = \frac1k {n \choose k}  x^{n-k}(1-x)^{k-1}$ and $B_2 = {
  n \choose k}I(x)$ where $I$ is as in the previous section. With these definitions $\E(s_k(t))=xB_1+B_2$ (the reason for taking one $x$ out of $B_1$ is technical ---
it allows us to use the same notations for the estimate of $B_1$
and of $B_2$, nothing more). Another useful notation would be
$$
\psi_{n,k}(y) = (n-k)\log(y) + (k-1) \log(1-y),
$$
so that the $x$-dependent term in $B_1$ is $e^{\psi(x)}$ and the integrand in $I(x)$ is $e^{\psi(y)}$. Note that $\psi_{n,k}$ has a unique maximum over $[0,1]$, attained at $y_{n,k} = (n-k)/(n-1)$. The key lemma is the following estimate:

\begin{lemma}\label{lem:B1}
There exists $C>0$ such that for all $k \le n$, and for all $x \in [0,1]$ we have:
  \begin{equation}
     B_1 \le \frac {Cq}k \exp\left\{ -
     \frac{\min\{|\eps|,\nicefrac 14\}^2 (n-k)}4\right\}\qquad\eps=\frac{x-y_{n,k}}{y_{n,k}},
  \end{equation}
where $q=n^{3/2}k^{-3/2}(n-k)^{-1/2}$, as in the statement of theorem \ref{T:trans}. 
\end{lemma}

\def\ynk{y_{n,k}}
\def\pnk{\psi_{n,k}}

\begin{proof}
We start by defining
\begin{equation}
\Phi_{n,k}(\eps) := \exp\big(\pnk(\ynk(1+\eps))\big)
\end{equation}
and then
\begin{align*}
\Phi_{n,k}(\eps) & = \exp\left\{ (n-k) \log\left(\frac{n-k}{n-1}(1+\eps)\right) + (k-1) \log \left(1- \frac{n-k}{n-1}(1+\eps)\right) \right\}\\
& = \exp\left\{( \psi_{n,k}(y_{n,k}) + \delta \right\}.
\end{align*}
To bound $\delta$ we first assume that $|\eps|\le\nicefrac 14$ and get
\begin{align}
\delta  &= (n-k) \log(1+ \eps) + (k-1) \log\left(1 - \eps\frac{n-k}{k-1}\right)\le\nonumber \\
& \le (n-k) \eps - (n-k) \frac{\eps^2}4 + (k-1) (- \eps) \frac{n-k}{k-1}\nonumber \\
& \le - \frac{(n-k) \eps^2}{4}, \label{psiUB}
\end{align}
where we have used for the first log the fact that for $|\eps| \le
1/4$, $\log(1+ \eps) \le \eps - \eps^2/4$, and for the second log
simply $\log(1+a)\le a$. To remove the restriction on $\eps$, note that $\psi_{n,k}$ is monotone
increasing over $[0,y_{n,k}]$ and monotone decreasing over
$[y_{n,k},1]$. So we can write $\delta\le-\frac
14(n-k)\min\{|\eps|,1/4\}^2$, for all $\eps$.

On the other hand, from Stirling's formula it is easy to see that there exists a universal constant $C$ such that for all $n$ and all
$k \le (n/2)$,
\begin{equation}\label{stirling}
{n \choose k} \le C\sqrt{\frac{n}{k(n-k)}}\frac{n^n}{k^k (n-k)^{n-k}} .
\end{equation}
Combining \eqref{stirling}
and \eqref{psiUB} we get
\begin{align}
  {n\choose k}\Phi_{n,k}(\eps) & \le C\sqrt{\frac{n}{k(n-k)}} \exp\left\{ k
  \log\left(\frac nk\right) +(n-k) \log\left(\frac n{n-k}\right) \right. \nonumber  \\
  & \hspace{2cm} \left.+ (n-k) \log\left(\frac{n-k}{n-1}\right) + (k-1) \log \left(\frac{k-1}{n-1}\right) + \delta\right\}\nonumber  \\
  & = C\sqrt{\frac{n}{k(n-k)}}\cdot\frac{n-1}{k-1} \exp\left\{ n \log\left(
  \frac{n}{n-1} \right) + k \log\left(\frac{k-1}{k} \right)  +\delta \right\}\nonumber \\
  & \le \frac{Cn^{3/2}}{k^{3/2}(n-k)^{1/2}} \exp \left\{ -\frac{ \min\{|\eps|,\frac 14\}^2 (n-k) }4\right\}\label{Phink}
\end{align}
Multiplying further by $1/k$, we get the
claimed bound on $B_1$ in Lemma \ref{lem:B1}. 
\end{proof}

\begin{lemma}
  \label{lem:B2}
    When $x> y_{n,k}(1+\eps)$ and $\eps>0$
  \begin{equation}
    B_2 \le Cq \exp\left\{ - \frac{\min\{|\eps|,\nicefrac 14\}^2 (n-k)}4\right\}
  \end{equation}
  while if $x< y_{n,k}(1 + \eps)$ and $\eps<0$,
  \begin{equation}
    \Big|B_2 - \frac1k\Big| \le Cq \exp\left\{ - \frac{\min\{|\eps|,\nicefrac 14\}^2 (n-k)}4\right\}.
  \end{equation}
where $q=q_n(k)$ is as in Lemma \ref{lem:B1} and Theorem \ref{T:trans}.
\end{lemma}

\begin{proof}
We analyze these two cases separately:

\medskip \noindent \textbf{Case 1:} $x> y_{n,k}(1+ \eps)$ with $\eps>0$. Then, since $\psi_{n,k}$ is monotone decreasing after $y_{n,k}$, we have
$$
I(x) = \int_x^1 e^{\psi_{n,k}(y)} dy \le  e^{\psi_{n,k}(y_{n,k}(1+\eps))},
$$
and hence
\begin{equation}\label{B2case1}
B_2 \le {n \choose k} e^{\psi_{n,k}(y_{n,k}(1+\eps))} ={n\choose k} \Phi_{n,k}(\eps).
\end{equation}

\medskip \noindent \textbf{Case 2:} $x<y_{n,k}(1+ \eps)$ with
$\eps<0$. As explained in the introduction, here we will Lemma
\ref{lem:0xx1}, which gives us
%
%
\begin{align*}
B_2 = { n \choose k}I(x) & = \frac1k  - {n \choose k} \int_0^x e^{\psi_{n,k}(y)}dy
\end{align*}
and this time, since $\psi_{n,k}$ is monotone {increasing} over $[0,y_{n,k}(1+ \eps)]$ we conclude that
$$
\int_0^x e^{\psi_{n,k}(y) }dy \le e^{\psi_{n,k}(y_{n,k}(1+\eps))}.
$$
so that
\begin{equation}\label{B2case2}
|B_2 - 1/k| \le {n\choose k}\Phi_{n,k}(\eps)
\end{equation}
Hence in both cases \eqref{B2case1} and \eqref{B2case2}, we conclude by \eqref{Phink}.
\end{proof}

\def\tnk{t_{n,k}}
With these two lemmas, we are now able to finish the proof of Theorem \ref{T:trans}.
\begin{proof}
  Suppose $t=t_{n,k} +\eps$, with $\eps \in \mathbb{R}$. Then
  \[
  x = e^{- t k} = e^{-k t_{n,k}} e^{ - k \eps}
    = \ynk e^{- k \eps}
  \]
  so
\begin{equation}\label{E:xyeps}
\frac{|x-y_{n,k}|}{y_{n,k}}=|e^{-k\eps}-1|>c\min\{|k\eps|,1\}=c\min\{k|t-t_{n,k}|,1\}.
\end{equation}
Thus
\begin{align*}
\left|\E(s_k(t))-\frac 1k\mathbf
1_{\{t>t_{n,k}\}}\right|&=\left|xB_1+B_2-\frac 1k\mathbf
1_{\{t>t_{n,k}\}}\right|\\
\mbox{By lemmas \ref{lem:B1} and \ref{lem:B2}}\qquad
  &\le Cq\exp\left(-\frac{\min\{|x-y_{n,k}|/y_{n,k},1/4\}^2(n-k)}{4}\right)\\
\mbox{By (\ref{E:xyeps})}\qquad
  &\le Cq\exp\big(-c\min\{k|t-t_{n,k}|,1\}^2(n-k)\big).
\end{align*}
As needed.
\end{proof}

\subsection{Proof of Theorem \ref{T:gc}}

We now turn to the proof of Theorem \ref{T:gc}. Let $t = c/n$ and let $X_n(\eps) = \frac1n \sum_{k \ge n \eps} k s_k(t)$ be the relative mass of cycles greater than $n \eps$. Then we have the following lemma:

\begin{lemma}\label{lem:massgc} Assume $c>1$ and let $\theta(c) = \max \{ z>0: 1-z > e^{-z c}\}$. Then $\theta(c)>0$ and
$$
\liminf_{n \to \infty }\E(X_n(\eps)) \ge \theta(c) - \eps.
$$
\end{lemma}

\begin{proof}
  Fix $\alpha < \theta(c)$ and let $\eps n \le k \le \alpha n$. It
  follows that
\begin{equation}\label{E:tnkcn}
t_{n,k}<\frac{c-\delta}{n}\quad \textrm{for some $\delta>0$.}
\end{equation}
Indeed, $k/n\le \alpha <\theta(c)$ gives that
$1-\nicefrac kn >e^{-(c-\delta)(k/n)}$ and with
\[
e^{-kt_{n,k}}\stackrel{\textrm{def}}{=}\frac{n-k}{n-1}>1-\frac kn>e^{-(c-\delta)(k/n)}
\]
we get (\ref{E:tnkcn}).
  Hence by Theorem \ref{T:trans}, noting that $q_n(k) \le C$ for any $k $ in this range,
  $$
  \E(s_k(t)) > \frac{1}{k} - C\exp( - \delta' n)
  $$
  for some $C, \delta'>0$ depending on $\alpha $ and $c$. Summing up between $k=\lceil \eps n \rceil$ and $k = \lfloor \alpha n \rfloor$, we get
  $$
  \E(X_n(\eps)) > (\alpha - \eps) - Cn^2 \exp( - \delta' n)
  $$
  The lemma follows immediately by letting $n \to \infty$, since $\alpha$ was arbitrarily close to $\theta(c)$.
\end{proof}

\begin{lemma}
  \label{lem:compRG} For any $\eps>0$, on an event of probability tending to 1 as $n\to \infty$,
  $$
  X_n(\eps) \le \theta(c) + \eps.
  $$
\end{lemma}

\begin{proof}
  Let $G(t)$ be the (random) graph where there is an edge $(i,j)$ if and only if the transposition $(i,j)$ has occurred prior to $t$. Then $G(t)$ is a realization of $G(n,p)$ with $p = 1- \exp( - c/n) \sim c/n$. Furthermore, it is well-known and easy to see that any cycle of $\sigma_t$ is a subset of some connected component of $G(t)$. By the Erd\H{o}s-Renyi theorem (see, e.g., \cite{ErdosRenyi1, ErdosRenyi2}), it is known that on an event of high probability, $G(t)$ has only one component greater than $\eps n$ (the giant component) and thus $X_n(\eps) \le Y_n$, where $Y_n$ is the rescaled size of the largest component.  Furthermore, the same papers show
  $$
  \frac{Y_n}n \to \xi(c),
  $$
  in probability, where $\xi(c)$ is the survival probability of a
  Poisson $(c)$ Galton-Watson process. Hence $\xi(c) = \theta(c)$ and
  the lemma follows.
\end{proof}

\begin{proof}[Proof of Theorem \ref{T:gc}] Recall that the theorem
  states that a cycle $>\eps n$ exists with high probability. Consider the random variable $ E_n = \theta(c) + \eps - X_n(\eps)$. Then $E_n$
  is bounded by 2, nonnegative with high probability (from the
  Galton-Watson argument), and has expectation bounded by $2\eps$
  (from Lemma \ref{lem:massgc}).
Hence by Markov's inequality $\P(E_n > \sqrt{\eps}) \le 3 \sqrt{\eps}$
for all $n$ sufficiently large. On the complement event $X_n(\eps) >
\theta(c) - \sqrt{\eps}$ and hence $X_n(\eps)>0$, in which case the
length of the longest cycle is $> n \eps$. Since $\eps$ was arbitrary, Theorem \ref{T:gc} follows.
\end{proof}

\subsection{Proof of Theorem \ref{T:slowdown}}

Recall that we wish to estimate $d(t)$, the graph distance to
$\sigma(t)$. Let $N(t)$ denote the number of cycles of $\sigma_t$. It is well known and easy to check that $d(t) = n - N(t)$. Let $\eps>0$ and fix $K \ge 1/\eps$. Then observe that, since there can never be more than $n\eps$ cycles greater than $K$,
$$
\frac1{n}\E(N(t)) = \frac1n \sum_{k=1}^K \E(s_k(t)) + O( \eps).
$$
Thus Theorem \ref{T:slowdown} follows from the claim: for all fixed $k \ge 1$, and all $c>0$, if $t = c/n$,
\begin{equation}\label{claimsk}
  \frac1n \E(s_k(t)) \to \frac{k^{k-2}}{k!} \frac1c (ce^{-c})^k.
\end{equation}
For this we apply Theorem \ref{T:exact} with these values of $k$ and
$t$, and note that $x = e^{-kt}$ satisfies $x \to 1$, $1- x \sim k
c/n$ and $x^{n-k} \to e^{-ck}$ (the notation $a\sim b$ is short for
$a=b(1+o(1))$, here and below). Since ${n \choose k} \sim n^k / k!$,
we deduce immediately that (recall that $\phi(x)=x^{n-k}(1-x)^{k-1}$),
$$
\frac1k {n \choose k} x \phi(x) \sim n \frac{k^{k-2}}{k!} \frac1c (ce^{-c})^k.
$$
Moreover, it is plain that $\frac1n { n \choose k} \phi(y)$ is uniformly bounded for $x \le y \le 1$, hence
$$
\frac1n{ n \choose k} \int_x^1 \phi(y)dy \le C(1-x) \to 0.
$$
\eqref{claimsk} follows immediately, hence so does Theorem \ref{T:slowdown}.



\end{document}